\documentclass{article}
\usepackage{amsthm, amsfonts, amssymb, latexsym, color, tikz}
\usepackage{hyperref}
\usepackage{lipsum}
\usepackage{fixme}
\fxsetup{
    status=draft,
    author=,
    layout=margin,
    theme=color
}

\newcommand{\E}{\mathsf{E}}
\newcommand{\F}{\mathbb{F}}

\newcommand{\R}{\mathbb{R}}
\usepackage{amsmath}
\usepackage[a4paper, total={6in, 8in}]{geometry}
\title{  On sum sets of convex functions}

\newcommand{\Addresses}{{
  \bigskip
  \footnotesize
   S. Stevens, \textsc{Johann Radon Institute for Computational and Applied Mathematics (RICAM), Linz, Austria}\par\nopagebreak
  \textit{E-mail address}: \texttt{sophie.stevens@oeaw.ac.at}
  
  \medskip

   A. Warren, \textsc{Johann Radon Institute for Computational and Applied Mathematics (RICAM), Linz, Austria}\par\nopagebreak
  \textit{E-mail address}: \texttt{audie.warren@oeaw.ac.at}

}}

\author{Sophie Stevens and Audie Warren}
\newtheorem{proposition}{Proposition}
\newtheorem*{proposition*}{Proposition}

\newtheorem{lemma}{Lemma}
\newtheorem*{corollary*}{Corollary 5}

\newtheorem{theorem}{Theorem}
\newtheorem{corollary}{Corollary}
\newtheorem{conjecture}{Conjecture}

\begin{document}

\maketitle
\begin{abstract}
    In this paper we prove new bounds for sums of convex or concave functions. Specifically, we prove that for all $A,B \subseteq \mathbb R$ finite sets, and for all $f,g$ convex or concave functions, we have
    $$|A + B|^{38}|f(A) + g(B)|^{38} \gtrsim |A|^{49}|B|^{49}.$$

    This result can be used to obtain bounds on a number of two-variable expanders of interest, as well as to the asymmetric sum-product problem. We also adjust our technique to also prove the three-variable expansion result
    \[
    |AB+A|\gtrsim |A|^{\frac32 +\frac3{170}}\,.
    \]
    
    Our methods follow a series of recent developments in the sum-product literature, presenting a unified picture. Of particular interest is an adaptation of a regularisation technique of Xue \cite{xue}, originating in \cite{rss}, that enables us to find positive proportion subsets with certain desirable properties. 

\end{abstract}
\section*{Introduction}
Given finite sets $A$ and $B$ of real numbers, the \emph{sum set} and \emph{product set} of $A$ and $B$ are defined as
$$A+B = \{ a+b : a\in A, b\in B \}, \quad AB = \{ ab : a\in A, b\in B \}.$$
Erd\H{o}s and Szemer\'{e}di conjectured that at least one of $|A+A|$ or $|AA|$ is large with respect to $|A|$. Specifically, they conjectured the following.\footnotemark
\begin{conjecture}[Erd\H{o}s - Szemer\'{e}di] \label{conj:erdsze}
For all $A \subseteq \mathbb Z$ a finite set, and for all $\epsilon > 0$, we have
$$|AA| + |A+A| \gg |A|^{2 - \epsilon}.$$
\end{conjecture}
\footnotetext{In this paper we use the standard notation $X \ll Y$ to mean that there exists an absolute constant $c$ with $X \leq cY$. We have $Y \gg X$ iff $X \ll Y$. We write $X\sim Y$ to denote the existence of constants $0 < c_1\leq c_2 $ so that $c_1X \leq Y \leq c_2X$. 
Additionally, the symbols $\lesssim$ and $\gtrsim$ are used to suppress logarithmic factors.} This conjecture remains open, and has given rise to the study of the \emph{sum-product} phenomenon, which, loosely defined, is the notion that finite sets cannot be simultaneously additively and multiplicatively structured. Conjecture \ref{conj:erdsze} is believed to be true over the real numbers, where current progress is given by Rudnev and Stevens \cite{MishaSophie}.

There are many variants of this problem in the literature; one family of such variants are concerned with convex functions\footnote{In this paper all convex functions considered are \emph{strictly} convex functions. Furthermore, our results also apply to strictly \emph{concave} functions.}. Such results quantify the notion that \emph{convex functions destroy additive structure.}
Some examples of common problems in this area are the following: 

For $A \subseteq \mathbb R$ a finite set, and $f$ be a convex function:
\begin{itemize}

\item Is the set $A + f(A)$ always large?
    \item Is at least one of the sets $A+A$ or $f(A) + f(A)$ is always large?
    
\end{itemize}
Much research has been done towards these problems and their variants, see for instance \cite{elekesnathansonruzsa,higherconvexity, liolly}. This is also related to the notion of a \emph{convex set}, that is, a set $A = \{ a_1 < a_2 < ... <a_n \}$ such that $a_{i+1} -a_i > a_i - a_{i-1}$ for all $2 \leq i \leq n-1$. Any convex set is the image of the interval $[n]$ under some convex function $f$. Current progress for these problems is given, respectively, by Li and Roche-Newton, \cite{liolly} and  Shkredov \cite{shkredovconvex}.

\begin{theorem}[Li, Roche-Newton] \label{convex2original}
Let $A \subseteq \mathbb R$ be a finite set, and let $f$ be a convex function. Then we have
$$|A + f(A)| \gtrsim |A|^{\frac{24}{19}}.$$
\end{theorem}
\begin{theorem}[Shkredov] \label{convexoriginal}
Let $A \subseteq \mathbb R$ be a finite set, and let $f$ be a convex function. Then we have
$$|A+A| + |f(A) + f(A)| \gtrsim |A|^{\frac{100}{79}}.$$
\end{theorem}

These problems are also related to \emph{expander} results. Results of this nature state that some set, defined by (typically polynomial) combinations of elements of $A$, is \emph{always} large. Two of the simplest examples of expanders are the sets
$$AA +A = \{ ab+c : a,b,c \in A \}, \qquad A(A+1) = \{a(b+1) : a,b \in A \}$$
which are both expected to have size $|A|^{2 - \epsilon}$ for all $\epsilon>0$. In fact, the expander $A(A+1)$ is a special case of the set $A + f(A)$ from above. The current bounds in the literature for these expanders are due to Roche-Newton and Warren \cite{ollyaudie} and Jones and Roche-Newton \cite{jonesolly}, respectively.
\begin{theorem}[Roche-Newton, Warren]\label{AA+Aoriginal}
For all $A \subseteq \mathbb R$ finite, we have
 $$|AA+A| \gtrsim |A|^{\frac{3}{2} + \frac{1}{194}}.$$
\end{theorem}
\begin{theorem}[Jones, Roche-Newton]\label{A(A+1)original}
 For all $A \subseteq \mathbb R$ finite, we have
$$|A(A+1)| \gtrsim |A|^{\frac{24}{19}}.$$
 \end{theorem}

\subsection*{Main results}
The proof of the sum-product result in \cite{MishaSophie} makes use of a combination of techniques used previously in the real numbers, combined with a technique used to prove sum-product results in finite fields, see \cite{rss}. In this paper we extend these techniques to give both quantitative and qualitative improvements to the problems mentioned above. Note that we make no attempt to optimise the logarithmic factors in our results, since in all cases the polynomial factor exponents are not expected to be tight. Our main result is the following.

\begin{theorem}\label{thm:main}
Let $A, B \subseteq \mathbb R$ be finite sets, and let $f$ and $g$ each be either a convex or concave function. Then we have
$$|A+B|^{38}|f(A) + g(B)|^{38} \gtrsim (|A||B|)^{49}.$$
\end{theorem}

For certain choices of $A,B,f,$ and $g$, this theorem implies improvements to many of the problems mentioned above. Firstly, we can recover the following improvements to Theorems \ref{convex2original} and \ref{convexoriginal}.

\begin{corollary} \label{cor:convex functions}
For all $A \subseteq \mathbb R$ finite, and $f$ a convex function, we have
$$|A + f(A)| \gtrsim |A|^{49/38},$$
$$|A + A| + |f(A) + f(A)| \gtrsim |A|^{49/38}.$$
\end{corollary}

The first inequality follows from setting $B = f(A)$ and $g = f^{-1}$. The second follows from setting $B = A$ and $f = g$. By slightly adjusting the proof of Theorem \ref{thm:main}, we can obtain a better bound for differences.

\begin{corollary} \label{cor:convex diff}
For all $A \subseteq \mathbb R$ finite, and $f$ a convex function, we have
$$|A - A|^5|f(A) - f(A)|^5 \gtrsim |A|^{13}.$$
\end{corollary}

In the case $A = [1,\dots,n]$, this matches the bounds of Schoen and Shkredov \cite{schoenshkredov} and Rudnev and Stevens \cite{MishaSophie} for estimates on differences and sums of convex sets respectively. Furthermore we match the result of Li and Roche-Newton \cite{liolly} in the case of few differences, many convex differences.

Secondly, we find an asymmetric sum-product result.
\begin{corollary}\label{sum-productasym}
For all $A,B \subseteq \mathbb R$ finite, we have
$$|AB|^{38}|A+B|^{38} \gtrsim (|A||B|)^{49}$$
\end{corollary}

This follows from setting $A = X$, $B = Y$, $f = g = \log(x)$. 
Corollary~\ref{sum-productasym} appears to be a little studied variant of the asymmetric sum-product problem: One example of a result in this direction is by Solymosi \cite{Solymosi}, who showed that $|A+A||B+B||AB|\gtrsim |A|^2|B|^2$. There has also been work towards the more difficult problem of finding a lower bound on $|A+B||AC|$, see for instance \cite{elekesnathansonruzsa}, or \cite[Theorem 10]{brendan-richlinesgrids},
where the results are rather of a qualitative nature. 
The statement of Corollary~\ref{sum-productasym} is particularly interesting in the extremal cases of `few sums' or `few products': e.g. if $|A|=|B|=N$ and $|A+B|\lesssim N$, then $|AB|\gtrsim N^{\frac32 + \frac3{38}}$. Typically the exponent of $3/2$ is a barrier in sum-product estimates, and so in this sense, Corollary~\ref{sum-productasym} is threshold-breaking.

Thirdly we give some results demonstrating the principle that
`translation destroys multiplicative structure', in particular improving Theorem \ref{A(A+1)original}.

\begin{corollary} \label{cor:A(A+1)}
For all $A,B \subseteq \mathbb R$ finite, we have
$$|A(A+1)| \gtrsim |A|^{49/38},$$
$$|AB| + |(A+1)(B+1)| \gtrsim (|A||B|)^{49/76}.$$
\end{corollary}

Finally, by combining techniques used in the proof of Theorem \ref{thm:main} with the method of Roche-Newton and Warren, we can give an improvement and generalisation of Theorem \ref{AA+Aoriginal}.

\begin{theorem}\label{thm: AB+A}
 Let $A,B \subseteq \mathbb R$ be finite sets with $|A| \sim |B|$. Then we have
 \[ |AB+A| \gtrsim |A|^{\frac{3}{2} + \frac3{170}}.\]
 \end{theorem}

\subsection*{Techniques}
Here we give an overview of the techniques that we use, hinting at the aspects of our method that are most amenable to future improvements.
These techniques can be summarised as follows:
\begin{enumerate}
    \item The Szemerédi-Trotter theorem gives good bounds on $\E_3^+(A,B)$, especially if we have data of the form $r_{QR}(a)\geq T$ for each $a\in A$. Similarly, 
    the Szemerédi-Trotter theorem gives good bounds on $\E_3^+(f(A),B)$ for a convex function $f$, if we have data of the form $r_{Q-R}(a)\geq T$ for each $a\in A$.
    \item Using a regularisation result, we can find a subset $C\subseteq A$ so that $|C|\gtrsim |A|$ and for which we have the additive data $r_{Q-R}(a) \geq T$ for each $c\in C$.
    \item We can count solutions $(a,b,c)$ to a tautological equation of the form $a-b = (a + c) - (b + c)$, where we insist that $a-b,a + c$ are in certain (different) sets via third moment energy bounds. This gives an auxiliary energy bound, see Proposition \ref{prop: energy general} below.
    \item A corollary of the regularisation result (see Corollary \ref{c:E_3(B,V)E_3(f(C),U)} below) allowing us to upper bound certain products of energies, together with this auxiliary energy bound, leads to the result.
    
\end{enumerate}

Underlying many results about expander sets in $\mathbb{R}$ (with few variables) is the Szemerédi-Trotter theorem. It is common knowledge that the Szemerédi-Trotter theorem is particularly strong for finding bounds on the third moment energy $\E_3^+(A,B)$, an idea first introduced by Schoen and Shkredov \cite{schoenshkredov}. This, in part, is due to the `trick' that every element of $A$ can be written as a product of elements of $AA$ and $A$ in at least $|A|$ ways: $a = (ab)/b$ for any choice of $b \in A$ (we assume here that $0\notin A$). However, if one has additional multiplicative structure on $A$, say $r_{QR}(a)\geq T$ for each $a\in A$ and some auxiliary sets $Q $ and $R$ and a number $T$, one can use this information in place of the aforementioned `trick'. This gives a third moment energy bound in terms of $Q,R$ and $T$, the strength of which depends on the strength of the multiplicative information. This is the idea behind the so-called \emph{Szemerédi-Trotter sets} introduced by Shkredov \cite{shkredovconvex}, for which the notation $d^+(A)$ (and variants thereof) is used.  We note that an analogue of this idea takes place in $\mathbb{F}_p$ using the point-line incidence bound of Stevens-de Zeeuw \cite{sophiefrank} in place of the Szemerédi-Trotter theorem, which naturally produces a bound on the fourth moment energy. For a convex function $f$, this trick changes as follows: we can obtain bounds on $\E_3^+(f(A),B)$ if we have \emph{additive} structure on $A$, say $r_{Q-A}(a)\geq T$ for all $a\in A$.

To benefit from the `enhanced energy trick' described above, we need the appropriate data on $Q, R$ and $T$. A generic technique for this, first described in \cite{energy} and refined in \cite{shkredovbw}, yields a subset $C\subseteq A$ with suitable parameters: that is, if $\E_3^+(A)\sim |D_t|t^3$ for some $D_t\subseteq A-A$, then $r_{D_t - A}(c)\geq |D_t|t|A|^{-1}$, and $|C|\geq |D_t|t|A|^{-1}$.  A recent expository lemma of Xue \cite{xue} enhances the strength of this result, to enable one to take $|C|\gtrsim |A|$ - we use an adaptation of this regularisation result. 



We conclude this section by considering where improvements to these techniques may be found. Certainly for the real numbers, there is hope that one could find a more optimised subset of $A$, with the data on $Q,R$ and $T$ optimised for the specific applications within our paper. Indeed, such a `better subset' is present in the current bounds for the sum-product problem \cite{MishaSophie}). In \cite{MishaSophie}, an elementary, somewhat geometric, argument justifies the existence of the subset used in the context of the sum-product problem. 

The third item of our list might also be improved as follows: we bound the number of solutions to $a-b = (a+c)- (b+c)$ in terms of the third moment energy. During this argument, we use Cauchy - Schwarz to bound a factor of $\E^+_{3/2}(A,\cdot)$ which appears as a by-product of H\"{o}lder's inequality. However, it may be possible to directly bound $\E^+_{3/2}(A,\cdot)$ using other methods. For example, if $A$ is a convex set, then Solymosi and Ruzsa \cite{solymosiruzsa} show that $\E^+_{3/2}(A,B) \ll |A+B|^{3/2}$ for any set $B$. 

In the proof of Theorem \ref{thm: AB+A}, we (implicitly) turn to the recent technique of studying the \emph{line energy} (see e.g. \cite{affinegroup4, ollyaudie}). We would not be surprised if future developments of this concept provide further tools relevant to the results in this paper.

\section{Preliminaries}
We use the notation $r_{Q-R}(a)$ to denote the number of representations of the element $a$ as a product from $Q-R$, that is, $r_{Q-R}(a) = |\{(q,r)\in Q\times r: q-r = a\}|$, and similarly for $r_{QR}(a)$ etc. The $k$th moment additive energy between sets $A$ and $B$ is defined to be 
\[
\E_k^+(A,B):= \sum_{x\in A-B}r_{A-B}^k(x)
\]
for $k\geq 1$. If $A=B$ we simply write $\E_k^+(A)$. Similarly, we define the multiplicative energy $\E_k^\times(A,B):= \sum_x r_{A/B}^k(x)$. 

\subsection{Energy Bounds via Szemer\'{e}di - Trotter}
Before beginning the proofs, we require some technical lemmas. The first gives a bound for the additive energy of two sets $A$ and $B$, subject to multiplicative information on the set $A$, and can be found in \cite{MishaSophie}. We give the proof for completeness, noting that the proof for Lemma \ref{lem: energy convex} follows from a similar argument. 

  \begin{lemma}\label{lem:estimatinge3}
Let $A, B, C,\, Q,R\subset \R$ be finite sets with  the property that $r_{QR}(a)\geq T$ for all $a\in A$ and some $T\geq 1$.
Then if $|R||C|\ll \left(|Q||B|\right)^2,$
\begin{equation} \label{e:bdone}
|\{c=a-b;\,a\in A, b\in B,\,c\in C\}| \ll \frac{(|Q||R||B||C|)^{2/3}}{T}\,.
\end{equation} 
Furthermore, if $|R||A|\leq |Q|^2|B|,$

\begin{equation}
    \label{e:e3(a,b)}
    \E_3 (A,B) \ll \frac{|Q|^2 |R|^2|B|^2}{T^3}\log|A|\,.
\end{equation}
\end{lemma}

The next lemma bounds the additive energy of two sets $f(A)$ and $B$, where $f$ is a convex function, subject to additive information on the set $A$.

\begin{lemma} \label{lem: energy convex}
 Let $A \subset \R$ be finite, and let $f$ be a convex (or concave) function. Suppose that there exist finite sets $Q, R\subseteq \R$ with $|Q|\geq |R|$ and a number $T\geq 1$ so that $r_{Q-R}(a)\geq T$ for all $a\in A$. Then for any set $B$ satisfying $|R||A|\ll  |Q|^2|B|$, we have
\[
\E_3^+(f(A),B) \ll \frac{|Q|^2|R|^2 |B|^2}{T^3}\log|A|\,.
\]
\end{lemma}

We remark that we have stated Lemmas~\ref{lem:estimatinge3} and  \ref{lem: energy convex} as a \emph{third} energy bound. The same technique with an additional interpolation argument gives us $k$th moment energy bounds, see e.g. \cite{MishaSophie} for details. 

\begin{proof}
To prove the first bound, we note that by utilising the information on the sets $Q$ and $R$, we have
$$|\{c = a-b: a \in A, b \in B, c \in C \} | \leq \frac{1}{T}|\{c = qr-b: q \in Q, r \in R, b \in B, c \in C \} |,$$

which can be viewed as incidences between the set of lines $L$ given by $y = qx - c$ for $(q,c) \in Q \times C$, and the point set $P = R \times B$. Applying the Szemer\'{e}di - Trotter theorem, we have
$$|\{c = qr-b: q \in Q, r \in R, b \in B, c \in C \} | = I(P, L) \ll \left(|Q||R||B||C| \right)^{2/3} + |Q||B|.$$

Because of the constraint present in the statement of the lemma, the leading term dominates. We therefore have
$$|\{c = a-b: a \in A, b \in B, c \in C \} | \ll \frac{\left(|Q||R||B||C| \right)^{2/3}}{T}$$
as needed. 

For the second part of the lemma, we decompose the support of  $\E_3^+(A,B)$ into dyadic groups: for $i= 0, \dots \lfloor \log|A|\rfloor$, let $D_i:=\{d\in A-B: r_{A-B}(d)\in [2^i,2^i)\}\subseteq A-B$. 
Then
\[
\E_3^+(A,B) = \sum_{i =0}^{\lfloor|A|\rfloor} \sum_{d\in D_i} r_{A-B}^3 (d) < \sum_i |D_i|2^{3i+3} \ll \log|A|\max_i |D_i|2^{3i}\,.
\]
With $D_i$ playing the role of $C$ in \eqref{e:bdone}, we have
\[
2^i |D_i|\leq |\{(a,b,d)\in A\times B \times D_i: d = a-b \}| \ll \frac{(|Q||R||B||D_i|)^{2/3}}{T}\,.
\]
The result then follows, and all that is left to do is to verify the condition required, for $C = D_i$, i.e. that $|Q||D_i| \ll \left(|R||B|\right)^2$. Note that since $D_i \subseteq A - B$, this is certainly true if we have
$$|Q||A||B| \ll \left(|R||B|\right)^2 \iff |Q||A| \ll |R|^2|B|$$
which is the stated condition.
\end{proof}

\subsection{Regularisation Results}
In this section we give some regularisation results required for the proof. The first is a lemma present in \cite{MishaSophie}. This lemma will be used to give a certain subset of $A$ on which much of the energy is supported, and with certain popularity properties.

\begin{lemma}\label{lem:regu}
Let ${\mathcal{R}_\epsilon}$ be a deterministic rule with parameter $\epsilon \in (0,1)$ that, to every sufficiently large finite additive set $X$, associates a subset ${\mathcal{R}_\epsilon}(X)\subseteq X$ of cardinality $|{\mathcal{R}_\epsilon}(X)|\geq (1-\epsilon)|X|$.

For any such rule ${\mathcal{R}_\epsilon}$, any $m>1$ and a sufficiently large finite set $A$,  set $\epsilon = c_1\log^{-1}(|A|)$ for some  $c_1\in (0,1)$. 
Then there exists a set $B\subseteq A$ (depending on ${\mathcal{R}_\epsilon}, \,m$), with $|B|\geq (1-c_1) |A|$ such that 
\[\E_m^+(\mathcal{R}_\epsilon(B))\geq c_2\,\E_m^+(B)\,,\]
 for some constant  $c_2=c_2(m, c_1)$ in $(0,1]$. \end{lemma}
 
We also require the following proposition. It is very similar to an expository lemma of Xue \cite[Lemma~5.1]{xue}, but has been amended to admit an asymmetric form. We present the rather technical proof of this proposition in the appendix, where we make the dependence on $\log(|A|)$ and $k$ hidden in the notation explicit. 

\begin{proposition}\label{p:decomp}
Let $A, V$ be finite subsets of $\mathbb R$, let $k> 1$ be a real number and fix $c_1 \in (0,1)$. 

Then there are sets $B, C$ with $C \subseteq B\subseteq A$ and $|C|\gtrsim_{k,c_1} |B|\geq (1-c_1) |A|$ 
such that the following property holds: there is a number $1\leq t\leq |B|$ and a set $D_t=\{x\in B-V:\, t\leq r_{B-V}(x)<2t\}$ such that 
\[
\E_k^+(B,V) \sim_k |D_t|t^k
\]
and 
\[
r_{D_t+V}(c) \sim_k \frac{|D_t|t}{|V|}
\]
for any $c\in C$.
\end{proposition}
 
On a high level, the proofs of Lemma~\ref{lem:regu} and Proposition~\ref{p:decomp} follow the same schemata: given a set $A$, we define a rule which extracts a positive proportion subset $A'\subseteq A$ with desirable properties according to the rule in question.  In Lemma~\ref{lem:regu}, this rule is abstract, whereas in Proposition~\ref{p:decomp} it is explicit. We then iterate this procedure until some stopping condition is satisfied. In Lemma~\ref{lem:regu}, this stopping condition is relative to the $m$th energy; in Proposition~\ref{p:decomp}, the stopping condition is defined with respect to the \emph{support} of the $k$th energy. These two regularisation results differ primarily because of this subtlety. Finally, we argue that this procedure must terminate in an acceptable number of steps, thus eventually outputting a positive proportion subset $B\subseteq A$. 

Proposition \ref{p:decomp} admits the following corollary, which is similar to a result of Shakan \cite[Theorem 1.10]{shakan}.

\begin{corollary}\label{c:E_3(B,V)E_3(f(C),U)}
Let $A, V\subseteq \R$ be finite, and $f$ be a convex (or concave) function. Then there are sets $B,C$ with $C\subseteq B\subseteq A$ and $|C|\gtrsim_k|B|\gg_k |A|$ such that 
\[
\E_3^+(B,V)\E_3^+(f(C),U)\lesssim |U|^2|V|^2|A|^3
\] 
for any set $U$ with $|U||V|\gg |A|$.
\end{corollary}
\begin{proof}
We apply Proposition~\ref{p:decomp} with $k = 3$ to obtain the sets $B$ and $C$, so that $\E_3^+(B,V)\sim |D_t|t^k$ where $r_{B-V}(d)\in [t,2t)$ for all $d\in D_t$ and $r_{D_t + V}(c)\gg |D_t|t|B|^{-1}$ for all $c\in C$.

We are able to obtain a bound on $\E_3^+(f(C),U)$ using Lemma~\ref{lem: energy convex}: 
\[
\E_3^+(f(C),U) \lesssim \frac{|D_t|^2|V|^2|U|^2}{|D_t|^3t^3|B|^{-3}} \sim \frac{|V|^2|U|^2|A|^3}{\E_3^+(B,V)}\,.
\]
We remark that since $|U||V|\gtrsim |A|$, it follows that $\min(|D_t|,|V|)|C|\lesssim \max(|D_t|,|V|)^2|U|$ and so we may indeed apply Lemma~\ref{lem: energy convex}.
\end{proof}

\section{Auxiliary energy bounds}
A unifying idea behind the proofs in this paper is the following proposition:

\begin{proposition}\label{prop: energy general}
Let $A,C\subseteq \mathbb{R}$ be finite, and $k\geq 1$. Suppose that $\E_k^+(A) \sim |D|\Delta^k$ for some $D\subseteq A-A$ and $\Delta \geq 1$, where $r_{A-A}(d) \in [\Delta,2\Delta)$. Then we have
\begin{equation}
    \label{e:prop eq}
|D|^9 \Delta^{12} \lesssim \frac{ |A + C|^6 \E_3^+(A)^4\E_3^+(C)^2\E_3^+(A,D) \E_3^+(C,A + C)^2}{|C|^{18}|A|^3}\,.
\end{equation}
\end{proposition} 
The stated form of Proposition~\ref{prop: energy general} gives us a great deal of flexibility. For example, if we had multiplicative information on the set $A$ in the guise of Lemma~\ref{lem: energy convex} -- that is, if $r_{QR}(a)\geq T$ for all $a\in A$ -- then we obtain an energy estimate in terms of this data. 
Proposition~\ref{prop: energy general} also admits a multiplicative form, in which $\E^\times_k(A) \sim |D|\Delta^k$. Then all instances of $\E_3$ in \eqref{e:prop eq} should be replaced by $\E_3^\times$, and $A+C$ by $AC$.
\begin{proof}
We begin by defining the popular set
\[P(A,C):= \left\{ x \in A+C : r_{A+C}(x) \geq \frac{|A||C|}{\log|A||A+C|} \right\}.\]

We also define the set 
\[A' :=\left\{ a\in A : |\{ c \in C : a+c \in P(A,C) \}| \geq \frac{|C|}{2} \right\}.\]

We perform a refinement step at the beginning of the proof, making use of Lemma \ref{lem:regu}. We claim that Lemma \ref{lem:regu} can be applied with the deterministic rule being the subset $A' \subseteq A$ defined above. Firstly we prove that $|A'|$ is large with respect to $|A|$. We have
 \begin{align*}\sum_{a \in A'}|\{ c \in C : a+c \in P(A,C) \}| + \sum_{a \in A \setminus A'}|\{ c \in C : a + c \in P(A,C) \}| &= |\{ (a,c) \in A\times C : a+c \in P(A,C) \}| \\
 & \geq \left( 1 - \frac{1}{\log|A|}\right) |A||C|.
 \end{align*}
 By setting $|A \setminus A'| = c|A|$ and using the bounds
 $$\sum_{a \in A'}|\{ c \in C : a + c \in P(A,C) \}| \leq (1-c)|A||C|$$
 $$ \sum_{a \in A \setminus A'}|\{ c \in C : a + c \in P(A,C) \}| \leq \frac{c}{2}|A||C|$$
 we conclude that $|A'| \geq \left( 1 - \frac{2}{\log|A|}\right) |C|$. We can therefore apply Lemma \ref{lem:regu} at the outset of the proof, obtaining a set $A'$ as above with the property that $|A'| \gtrsim |A|$, and $\E_k^+(A') \sim \E_k^+(A)$.
 
 We now consider the number of solutions $(a,b,c) \in A^2 \times C$ to the trivial equation
\begin{equation}\label{trivialeq} a-b = (a+c) - (b+c)\end{equation}
where the difference $a-b$ comes from the set $D \subseteq A'-A'$ such that $|D|\Delta^k \sim\E_k^+(A') \sim\E_k^+(A)$, and such that the sum $a + c$ is popular, that is, $a+c \in P(A,C)$.
 
 There are at least $\Omega( |C||D|\Delta)$ solutions to equation \eqref{trivialeq}. We partition solutions to \eqref{trivialeq} with the relevant conditions, via the following:
 $$(a,b,c) \sim (a+ t,b+ t,c-t), \quad t \in \mathbb R\,,$$
 and let $[a,b,c]$ represent this equivalence class. Since $t$ cancels out in equation \eqref{trivialeq},  these classes are non-trivial. 
 
 Let $N$ denote the number of solutions to equation \eqref{trivialeq}. 
 We have
 $$|C||D|\Delta \ll N = \sum_{[a,b,c]}|[a,b,c]|$$
and so, after an application of the Cauchy-Schwarz inequality, we obtain
 \begin{equation}\label{eq:twofactors} \left(|C||D|\Delta\right)^2 \leq |\{ \text{equivalence classes} \}| \cdot \sum_{[a,b,c]}|[a,b,c]|^2.\end{equation}
 
 We now aim to bound the two factors in equation \eqref{eq:twofactors}.
 
 To bound the number of equivalence classes, note that each equivalence class gives a solution to the equation
 $$d = s_1 - s_2, \quad d \in D, s_1 \in P(A,C), s_2 \in A+B.$$
 
 Therefore we have
 $$|\{  \text{equivalence classes} \}| \leq |\{  (d,s_1,s_2) \in D \times P(A,C) \times A + C : d = s_1-s_2 \}|.$$
 
By the popularity of $s_1$, we have
  \begin{align*}|\{  \text{equivalence classes} \}| & \lesssim \frac{|A+C|}{|A||C|}|\{  (d,a,c,s) \in D \times A \times C \times A + C : d - a =  c - s \}| \\
  & =  \frac{|A + C|}{|A||C|} \sum_{x} r_{D - A}(x) r_{ C - (A+C)}(x)
  \\
& \leq  \frac{|A + C|}{|A||C|}\E_3^+(A,D)^\frac16 (|A||D|)^\frac12\E_3^+(C,A+C)^\frac13\,,
  \end{align*}
where the final bound is an application of H\"older's inequality, followed by Cauchy-Schwarz. 

We now aim to bound the sum 
$$ \sum_{[a,b,c]}|[a,b,c]|^2$$
where is it understood that the sum is taken over equivalence classes satisfying the relevant conditions. Note that this sum counts pairs of triples from the same equivalence class, and for each pair we have
$$(a,b,c) \sim (a',b',c') \implies \exists
~t \text{ \ with \ }  a-a'= b- b'= c'- c = t.$$

We therefore have
$$ \sum_{[a,b,c]}|[a,b,c]|^2 \leq \sum_t r_{A-A}(t)^2r_{C-C}(t) \leq \E_3^+(A)^\frac23\E_3^+(C)^\frac13 $$
where again, the final inequality is a result of Hölder's estimate. 

Finally, from \eqref{eq:twofactors} we have
\[\left(|C||D|\Delta\right)^2 \lesssim 
 \frac{|A + C|}{|A||C|}\E_3^+(A,D)^\frac16 (|A||D|)^\frac12\E_3^+(C,A + C)^\frac13
 \E_3^+(A)^\frac23\E_3^+(C)^\frac13
\,.
\] 
Rearranging and raising both sides to the sixth power concludes the proof.
\end{proof}

\section{Proof of Theorem~\ref{thm:main}}
We actually prove the following, slightly more general theorem.
\begin{theorem}\label{t:|A+B||f(A)+g(B)|etc}
Let $f$ and $g$ be convex or concave functions. Let $A, B \subseteq \mathbb{R}$ be finite sets. Then we have
\begin{equation}\label{e:|A+B||f(A)+g(B)|}
|A|^{49}|B|^{49}\lesssim
|A+B|^{38} |f(A)+g(B)|^{38}
\end{equation}
and
\begin{equation}\label{e:|A+B||f(A)+g(B)||A-A|etc}
|A|^{39}|B|^{39}\lesssim |A\pm B|^{20}|f(A)\pm g(B)|^{20}|A-A|^5||B-B|^5|f(A)-f(A)|^5|g(B)-g(B)|^5\,.
\end{equation}
\end{theorem}
We clarify that in this theorem, one may take $f$ to be convex, and $g$ to be concave.

On a high level, the proof proceeds by apply  two iterations of Corollary~\ref{c:E_3(B,V)E_3(f(C),U)} to $A$ and $B$ with judicious choices of $V$ in each case. Then we apply Proposition~\ref{prop: energy general} to the ensuing subsets and their convex (resp. concave) counterparts. This gives an additive energy relation.
We obtain the statements of Theorem~\ref{t:|A+B||f(A)+g(B)|etc} using Cauchy-Schwarz and H\"older inequalities. 

Let us make two simple observations regarding the third moment energy of a set $X$ that we use in the subsequent argument.
 Firstly, note that $\E_3^+(-X,X-X) = \E_3^+(X,X-X)$ due to the symmetry of the difference set. Secondly, for any set $Y\subseteq X$ and any set $Z$, we have $\E_3^+(Y,Z) \leq \E_3^+(X,Z)$.
\begin{proof}
Here we prove the slightly more technical statement~\eqref{e:|A+B||f(A)+g(B)|}, and indicate the changes necessary to prove \eqref{e:|A+B||f(A)+g(B)||A-A|etc}.

We begin by applying Corollary~\ref{c:E_3(B,V)E_3(f(C),U)} to the set $A$ with $V = A$ to obtain sets
$A_2\subseteq A_1 \subseteq A$ with $|A_2|\gtrsim |A_1|\gg |A|$
and 
\[E_3(A_1,A)E_3(f(A_2),U) \lesssim |A|^5 |U|^2 \text{ for any }U\,.\]
Note that if $|U|\gg 1$, then this follows from Corollary \ref{c:E_3(B,V)E_3(f(C),U)}; if  $|U|\ll 1$, then it follows trivially.

We now apply the concave analogue of Corollary~\ref{c:E_3(B,V)E_3(f(C),U)}, this time to the set $f(A_2)$ with $V =f(A)$ and the function $f^{-1}$. We obtain \footnote{Strictly speaking, we obtain sets $f(A_4)\subseteq f(A_3)\subseteq f(A_2)$.} the sets $A_4\subseteq A_3\subseteq A_2$ with $|A_4|\gtrsim |A_3|\gg |A_2|\gtrsim |A|$
 so that 
\[E_3(f(A_3), f(A)) E_3(A_4,U) \lesssim |A|^5 |U|^2 \text{ for any }U\,.
\]

We repeat this argument for the set $B$ taking $V  = B$ to obtain $B_2\subseteq B_1\subseteq B$
so that
\[E_3(B_1,B)E_3(g(B_2),U) \lesssim |B|^5 |U|^2 \text{ for any }U\,,\]
and then once more to $g(B_2)$ with $V = g(B)$ and function $g^{-1}$ to obtain 
$B_4\subseteq B_3\subseteq B_2$ with $|B_4|\gtrsim |B|$ and
\[E_3(g(B_3), g(B)) E_3(B_4,U) \lesssim |B|^5 |U|^2 \text{ for any }U\,.\]

To prove \eqref{e:|A+B||f(A)+g(B)|}, we dyadically decompose the sets $A_4,B_4,f(A_4),g(B_4)$ according to the second moment energy to obtain sets $D_i$ and numbers $t_i\geq 1$  so that
\[
\E_2^+(A_4) \sim |D_1|t_1^2, \quad \E_2^+(B_4) \sim |D_2|t_2^2\]
\[\E_2^+(f(A_4)) \sim |D_3|t_3^2, \quad \E_2^+(g(B_4)) \sim |D_4|t_4^2\,.\]
 
To prove \eqref{e:|A+B||f(A)+g(B)||A-A|etc}, we would instead dyadically decompose the sets $A_4, B_4, f(A_4), g(B_4)$  according to the $12/7$th moment energy, so that e.g. $\E^+_{12/7}(A_4) \sim |D_1|t_1^{12/7}$. Note that e.g. $D_1 \subseteq A_4 - A_4$. 

We now apply Proposition~\ref{prop: energy general} to each of the sets $A_4,B_4,f(A_4),g(B_4)$, choosing $C$ in \eqref{e:prop eq} to be $B_4,A_4,g(B_4),f(A_4)$ respectively. 

We then multiply together the four instances of \eqref{e:prop eq}, and make liberal use of the simple observations noted at the beginning of this section together with the consequences of Corollary~\ref{c:E_3(B,V)E_3(f(C),U)}. This gives 
\begin{equation}\label{e:product 12/7 energy}
\prod_{1\leq i \leq 4}|D_i|^7t_i^{12} \lesssim|A+B|^{20} |f(A)+g(B)|^{20}|A|^9|B|^9\,.
\end{equation}

To prove statement \eqref{e:|A+B||f(A)+g(B)||A-A|etc}, we recall that we had initially dyadically decomposed according to the $12/7$th energy and so, after an application of H\"older's inequality for $\E^+_{12/7}(A_4)$ etc., we are done. 

To prove statement \eqref{e:|A+B||f(A)+g(B)|}, let us multiply \eqref{e:product 12/7 energy} on both sides by $(t_1t_2t_3t_4)^2$. 

Note that 
\[
|D_1|t_1^3|D_3|t_3^3 \leq \E_3^+(A_4,A)\E_3^+(f(A_3),f(A)) \lesssim |A|^7 \implies t_1t_2\lesssim \frac{|A|^7}{\E_2^+(A)\E_2^+(f(A))}
\]
and similarly
\[
(t_2t_4) \lesssim \frac{|B|^7}{\E_2^+(B)\E_2^+(g(B))}
\]\,.

Hence we obtain 
\[
\left(\E_2^+(B)\E_2^+(g(B))\E_2^+(A)\E_2^+(f(A))\right)^9 \lesssim|A+B|^{20} |f(A)+g(B)|^{20}|A|^{23}|B|^{23}\,.
\]
Finally, we use the Cauchy-Schwarz relation
\[
\frac{|X|^2|Y|^2}{|X+Y|}\leq \E_2^+(X,Y) \leq \E_2^+(X)^{1/2}\E_2^+(Y)^{1/2}
\]
to complete the proof. 
\end{proof}

\section{Proof of Theorem \ref{thm: AB+A}}

In this section we prove Theorem \ref{thm: AB+A}
proving two complementary bounds, using a combination of the methods found in \cite{rss}, \cite{MishaSophie}, and \cite{ollyaudie}. 
\subsection{Proof of Theorem~\ref{thm: AB+A} - Bound 1}
The method of Roche-Newton and Warren \cite{ollyaudie} involved studying the \emph{line energy} of lines of a particular structure. Their results, combined with an incidence theorem of Rudnev and Shkredov \cite{rudnevshkredov} and an additive combinatorial result of Roche-Newton and Rudnev\footnote{The result of Roche-Newton and Rudnev is that the number of solutions to the equation
\[
\frac{a_1 - a_2}{a_3 - a_4} = \frac{a_5 - a_6}{a_7- a_8}
\]
with each $a_i \in A$ is at most $O(|A|^6\log|A|)$.} \cite{Q}
 imply the following incidence bound. See also \cite{affinegroup4} for more information on line energy and its applications.

\begin{theorem}
\label{thm:incidence}
Let $L$ be a set of lines of the form $y = ax + a'$ for $a,a' \in A \subseteq \mathbb R\setminus\{0\}$ a finite set. 
Let $B,C \subseteq \mathbb R$ be two finite sets. Then we have
\[\mathcal{I}(B \times C, L) \lesssim \E_4^\times(A)^\frac1{12}|A|^\frac76
|B|^\frac23|C|^\frac12
+ |A|^2 |C|^\frac12.
\]
\end{theorem}

We shall apply Theorem \ref{thm:incidence} to the point set $B \times (AB+A)$ and to the set of lines $L$ of the form $y = ax + a'$ with $a,a' \in A$. Note that without loss of generality we may remove $0$ from $A$ if it is present. For each line $y = ax + a'$, for each $b \in B$ the point $(b,ab+a')$ lies on this line, and so we have at least $|A|^2|B| \sim |A|^3$ incidences. Using Theorem \ref{thm:incidence} we obtain 

\begin{equation}\label{Incidence1}
    |A|^3 \lesssim\mathcal I(B \times (AB+A),L) \lesssim
     \E_4^\times(A)^\frac1{12}|A|^\frac76
|B|^\frac23|AB+A|^\frac12
+ |A|^2 |AB+A|^\frac12.
\end{equation}

Note that if the second term dominates we have a much stronger result than claimed in the statement in Theorem~\ref{thm: AB+A}. Let us therefore assume the first term dominates. Hence we have the first of our two bounds: 
\begin{equation}\label{e:bound1}
|A|^{7/3}\lesssim |AB+A|\E_4^\times(A)^{1/6} \leq |AB+A||A|^{\frac{1}{6}} \E_3^\times(A)^{1/6}.
\end{equation}

\subsection{Proof of Theorem~\ref{thm: AB+A} - Bound 2}
To find the second bound, let us apply the multiplicative version of Proposition~\ref{p:decomp} to the set $A$ with $V = A$  to obtain sets $A_2\subseteq A_1\subseteq A$ with $|A_2|\gtrsim |A_1|\gg |A|$ so that
\begin{equation}\label{e:mult version of cor}
    \E_3^\times(A_1,A)\E_3^+(A_2,U)\lesssim |A|^5|U|^2\,.
\end{equation}
Equation~\eqref{e:mult version of cor} is a consequence of using Lemma~\ref{lem:estimatinge3} in place of Lemma~\ref{lem: energy convex} in the proof of Corollary~\ref{c:E_3(B,V)E_3(f(C),U)}. 

We now apply Proposition~\ref{prop: energy general} to the set $A_2$, writing $\E^+(A_2)\sim |D|t^2$ and taking $C = \lambda A_2$ for $\lambda \neq 0$. Note that $\E_k^+(A_2,X) = \E_k^+(\lambda A_2, \lambda^{-1} X)$ for any set $X$ and any $k\geq 1$. 

From Proposition~\ref{prop: energy general}, \eqref{e:mult version of cor}, and the inclusions $A_2\subseteq A_1\subseteq A$ we obtain
\begin{equation}\label{e: bound2 penultimate}
    |D|^7 t^{12} \lesssim \frac{|A+\lambda A|^{10}}{|A|^9} \left(\frac{\E_3^+(A_2)}{|A|^2}\right)^{12} \frac{\E_3^+(A_2,D)}{|D|^2} \left(\frac{\E_3^+(A_2,\lambda^{-1}A + A}{|A+\lambda A|^2}\right)^2 \lesssim \frac{|A+\lambda A|^{10}}{|A|^9} \frac{|A|^{45}}{\E_3^\times(A_1,A)^9}\,.
\end{equation}

We have 
\[|D|t^3 \leq \E_3^+(A_2) \lesssim \frac{|A|^7}{\E_3^\times(A_1,A)} \implies t \lesssim \frac{|A|^7}{\E_3^\times(A_1,A) \E_2^+(A_2)}\,,
\]
and so multiplying \eqref{e: bound2 penultimate} by $t^2$ and applying the Cauchy-Schwarz energy bound 
\[\frac{|A_2|^4}{|A_2 + \lambda A_2|} \leq \E(A_2,\lambda A_2) \leq \E(A_2)\]
we conclude that 
\begin{equation}
    \label{e:bound 2 final}
    |A+\lambda A|^{19} \gtrsim \frac{\E_3^\times(A_1,A)^{11}}{|A|^{14}}\,.
\end{equation}

\subsection{Proof of Theorem~\ref{thm: AB+A} - Conclusion }
Combining the bounds of the previous section we obtain 
\[
|A+BA|\gtrsim \max \left(
\frac{\E_3^\times(A,A_1)^\frac{11}{19}}{|A|^{\frac{14}{19}}}
, \frac{|A|^\frac{13}{6}}{ \E_3^\times(A,A_1)^\frac16}\right)
\]
where the first bound has instead been applied to the set $A_1$ 
given above, making use of the inequalities $|A| \sim |A_1|$ and $\E_3^+(A_1) \leq \E_3^+(A,A_1)$. In the worst possible case, both maximands are equal. This happens if 
\[ \E_3^+(A_1,A)  =  |A|^\frac{331}{85}\]
and so we shall assume that this is indeed the case. We then obtain
\[ |A+BA|\gtrsim  |A|^\frac{129}{85} = |A|^{\frac32 +\frac3{170}}\]
as required.

\section*{Acknowledgements}
The authors were supported by Austrian Science Fund FWF Project P 30405-N32. We are especially grateful to Oliver Roche-Newton who pushed us to improve our results. We also thank Misha Rudnev for helpful suggestions.

\Addresses
\newpage
\appendix
\section*{Appendix - Proof of Proposition~\ref{p:decomp}}
We present the proof of Proposition~\ref{p:decomp} in this section. 
\begin{proposition*}
Let $k> 1$ be a given {real number}. Fix $c_1 \in (0,1)$, and let $A, V$ be finite subsets of $\mathbb R$. Then there are sets $B, C$ with $C \subseteq B\subseteq A$ and $|C|\gtrsim_{k,c_1} |B|\geq (1-c_1) |A|$, such that the following property holds: there is a number $1\leq t\leq |B|$ and a set $D_t=\{x\in B-{V}:\, t\leq r_{B-{V}}(x)<2t\}$ such that 
\[
|D_t|t^k\leq E_k^+(B,{V})  \leq 2^k|D_t|t^k\log(|B|)\]
and
\[
r_{P+{V}}(c) \in \left[\frac{|D_t|t}{2^{k+1}|B|}, \frac{2|D_t|t}{|B|}\frac{k2^k \log^2|A|}{c_1} \right] \]
for any $c\in C$.
\end{proposition*}

Here, the subscripts in $\gtrsim_k$, $\gg_k$ or $\approx_k$ means that the implied constant may depend on $k$.

\begin{proof}
The proof of this lemma is two-fold: first we will refine the set $A$ iteratively according to a deterministic rule $\mathcal{R}_\varepsilon$ for a fixed $\varepsilon>0$. This yields a set $B\subseteq A$ so that $|B|\geq (1-\varepsilon)|A|$. We then choose  $C\subseteq \mathcal{R}_\varepsilon(B)$ and argue that the iteration process guarantees that $C$ has the desired properties.

Let us first describe the deterministic rule applied to a dummy set $\tilde A$. 
We begin by dyadically decomposing $\tilde A-V$ to obtain $D_t\subseteq \tilde A -V$ and a number $1\leq t \leq \min(|\tilde A|,|V|)$ so that $|D_t|t^k \leq \E_k^+(\tilde A,V) \leq \log(\min(|\tilde A|, |V|)) |D_t|t^k$ and $r_{\tilde A-V}(d)\in [t,2t)$ for any $d\in D_t$. That is, we perform a dyadic decomposition argument according to the $k$th energy. We also define the set 
\[\mathcal{P}_{\tilde A}:= \left\{(a,v)\in \tilde A\times V \colon a-v\in D_t \right\},\]
which is the set of points of $\tilde A\times V$ supported on lines with slope in $D_t$. By construction, we have that $|\mathcal{P}_{\tilde A}| \in[|D_t|t,2|D_t|t )$. Now, with $\varepsilon > 0$ a parameter, define  a subset 
    \[
    \mathcal{R}_\varepsilon(\tilde A) :=\left\{a\in \tilde A: r_{D_t+V}(a)\leq  \frac{|\mathcal{P}_{\tilde A}|}{ \varepsilon |\tilde A|}\right\}\,.
    \]
 Clearly this rule is deterministic; we claim that $|\mathcal{R}_\varepsilon (\tilde A)|\geq (1-\varepsilon) |\tilde A|$. Indeed, writing $\tilde A' = \mathcal{R}_\varepsilon(\tilde A)$, we have
 $$|\tilde A\setminus \tilde A'| \frac{|\mathcal{P}_{\tilde A}|}{ \varepsilon |\tilde A|} \leq |\{(a,v)\in \mathcal{P}_{\tilde A}: a\notin \tilde A'\}| \leq |\mathcal{P}_{\tilde A}|$$ and it follows that $|\tilde A'|>(1-\varepsilon)|\tilde A|$.

We now describe the iteration scheme in the proof using the notation introduced above: let $A_0 = A$ and for $i \geq 0$ define 
\[G_{i}:= \mathcal{P}_{A_i} ~\cap~ (\mathcal{R}_\varepsilon(A_i)\times V) = \{(a,v) \in \mathcal{R}_\varepsilon(A_i)\times V \colon a-v \in D_t(A_i)\}
\]
where $D_t(A_i)\subseteq A_i-V$ is the set supporting the $k$th energy $\E_k^+(A_i,V)$.

If $|G_i|<2^{-k}|\mathcal{P}_{A_i}|$ then set $A_{i+1} = \mathcal{R}_\varepsilon(A_i)$. Otherwise we terminate the process and set $B = A_i$ and $C' = \mathcal{R}_\varepsilon(A_i)$.

Note that $|A_i| \geq (1-\varepsilon)^i |A|\geq (1-i\varepsilon)|A|$.

We remark that the stopping condition of this algorithmic procedure differs from the stopping condition in Proposition~\ref{prop: energy general}. 

We claim that the this process must terminate in fewer than $I = c_1\varepsilon^{-1}$ steps. For ease of notation, let us suppose that $c_1\varepsilon^{-1}\in \mathbb{N}$.

Indeed, suppose for contradiction that we are in the $I$th step of the process. Then we have a set $A_I\subseteq A$ so that $|A_I|\geq (1-c_1)|A|$ and
for each $0\leq i\leq I$ we have
$|G_i|< 2^{-k} |\mathcal{P}_{A_i}|$.

Let us write $D_i$ to mean $D_t(A_i)$ - that is $\mathcal{P}_{A_i} = \{(a,v)\in A_i\times V\colon a-v\in D_i\}$. Similarly, let us write the $t$ corresponding to $D_i$ as $t_i$ so that $\E_k^+(A_i,V) \in [ |D_i|t_i^k, 2^k|D_i|t_i^k \log(|A_i|) )$.

Since we have not terminated the iteration procedure, we obtain for each $0\leq i \leq I - 1$ that 
\begin{align*}
\sum\limits_{x\in D_i} r_{\mathcal{R}_\varepsilon(A_i) -V}^k(x) &  \leq (2t_i)^{k-1}
\left |\mathcal{P}_{A_i}~\cap~(\mathcal{R}_\varepsilon(A_i) \times V)\right| = (2t_i)^{k-1} |G_i|  \\
&< 2^{-k} (2t_i)^{k-1} |\mathcal{P}_{A_i}| = t_i^{k-1}|\mathcal{P}_{A_i}|/2
\end{align*}

Let us now consider the number of terms in the support of the energy that we discard during the iteration process:
\begin{align*}
|\{
\left((a_1,v_1), \dots, (a_k,v_k)\right)\in (\mathcal{P}_{A_i}\setminus G_i)^k &\colon a_1 - v_1 = \dots = a_k - v_k  
\}| = \sum_{x\in D_i}r_{A_i-V}^k(x) - \sum_{x\in D_i}r_{\mathcal{R}_\varepsilon(A_i) - V}^k(x) \\
& \geq t_i^{k-1}|\mathcal{P}_{A_i}|/2  \geq  |D_i|t_i^k \geq 2^{-k} \E_k^+(A_i,V) \log(|A_i|)^{-1}
\end{align*}

We emphasise that any discarded energy-term $((a_1,{v_1}),\ldots,(a_k,{v_k}))$ has at least one component $(a_j,{v_j})$ with abscissa not in $\mathcal{R}_\varepsilon(A_i)$. So the energy-terms counted by $E_k^+(\mathcal{R}_\varepsilon(A_i), {V})$ all remain. We deduce that
\begin{align*}
E_k^+(A_{i+1},{V}) &= E_k^+(\mathcal{R}_\varepsilon(A_i), {V})\\
&\leq E_k^+( A_i, {V}) - 
|\left\{((a_1,{v_1}),\ldots,(a_k,{v_k}))\in (\mathcal{P}_{A_i}\setminus G_i)^k:\, a_1-{v_1}=\ldots =a_k-{v_k}\right\}|\\
&\leq (1 - 2^{-k} \log(|A_i|)^{-1}) \E_k^+(A_i,V)\\
&\leq (1 - 2^{-k} \log(|A|)^{-1}) \E_k^+(A_i,V)
\end{align*}
for all $0\leq i\leq I-1$. 

Using the trivial bounds $|A||V|\leq \E_k^+(A,V)\leq |A|^k|V|$ we obtain the bound
\[
(1-c_1)|A||V| \leq |A_I||V|\leq
\E_k^+(A_I,V)\]
and similarly 
\[
\E_k^+(A_I,V)\leq (1 - 2^{-k} \log(|A|)^{-1}) \E_k^+(A_{I-1},V) \leq  (1 - 2^{-k} \log(|A|)^{-1})^I \E_k^+(A,V) \leq (1 - 2^{-k} \log(|A|)^{-1})^I |A|^k|V|\,.
\]

Thus we have the estimate
\[
(1-c_1) < (1 - 2^{-k} \log(|A|)^{-1})^I |A|^{k-1} \leq e^\frac{-c_1}{\varepsilon 2^{k}\log|A|} |A|^{k-1} 
= e^{\frac{-c_1}{\varepsilon 2^{k}\log|A|} + \ln(2) (k-1)\log|A|}
\]

Let us choose $\varepsilon$ so that
\[
(1-c_1)  =  e^{\frac{-c_1}{\varepsilon 2^{k}\log|A|} + \ln(2) (k-1)\log|A|}
\]
to obtain a contradiction. 

That is , let us take 
\[
\varepsilon = \frac{c_1 }{2^k \ln(2) (k-1)\log^2|A| - 2^k\log|A|\ln(1-c_1)} \,.\]

With this choice of $\varepsilon$, the process must terminate in at most $I = c_1 \varepsilon^{-1}$ steps. 

Having argued that this algorithmic procedure must indeed terminate after say $N\leq I$ steps , let us set $B =A_N$   and $C' = \mathcal{R}_\varepsilon(B)$. We have that $|B|\geq (1-c_1)|A|$.
Set
\[
C = \{x\in C^\prime:\, r_{D_I+{V}}(a)\geq |\mathcal{P}_{B}|/(2^{k+1}|B|)\}.
\]

Then 
\[
|\{(a,{v})\in \mathcal{P}_{B}:\, a\in C^\prime\setminus C\}|  \leq \frac{|\mathcal{P}_{B}|}{2^{k+1}|B|} |C^\prime| \leq \frac{|\mathcal{P}_{B}|}{2^{k+1}}.
\]

Thus we obtain 
\[
|\{(a,{v})\in \mathcal{P}_{B}:\, a\in C^\prime\setminus C\}|  \geq  |G_N| -  \frac{|\mathcal{P}_{B}|}{2^{k+1}} \geq \frac{|\mathcal{P}_{B}|}{2^k} - \frac{|\mathcal{P}_{B}|}{2^{k+1}} = \frac{|\mathcal{P}_{B}|}{2^{k+1}}\,,
\]
where the second inequality is a consequence of the termination condition. 

On the other hand, since $C \subseteq C^\prime$, recalling the definition of $C^\prime$, we have
\[
|\{(a,{v})\in \mathcal{P}_{B}:\, a\in C\}| \leq |C|\frac{|\mathcal{P}_{B}|}{\varepsilon |B|}.
\]

Hence $|C|\geq \varepsilon2^{-(k+1)}|B|$. With the explicit choice of $\varepsilon$ together with the bound on $|B|$ this means that 
\[
 |C| 
 \geq \frac{c_1(1-c_1) }{2^{2k+1} \ln(2) (k-1)\log^2|A| - 2^{2k+1}\log|A|\ln(1-c_1)}|A|
 \geq \frac{c_1(1-c_1) }{2^{2k+1} \ln(2) (k-1)}\frac{|A|}{\log^2|A|}
\]
%

Note that for any $c\in C$ we have
\[
r_{P+{V}}(c) \in \left[\frac{|D_t|t}{2^{k+1}|B|}, \frac{2|D_t|t}{|B|}\frac{2^k \ln(2) (k-1)\log^2|A| - 2^k\log|A|\ln(1-c_1)}{c_1} \right] \,.\]
The upper bound is certainly less than 
\[
 \frac{2|D_t|t}{|B|}\frac{k2^k \log^2|A|}{c_1}\,,
\]
the bound that appears in the statement of the proposition. This completes the proof.

\end{proof}


\begin{thebibliography}{}

\bibitem{elekesnathansonruzsa} G. Elekes, M. Nathanson, and I. Ruzsa, \emph{Convexity and sumsets}, J. Number Theory, 83, 194--201 (1999)

\bibitem{higherconvexity}B. Hanson, O. Roche-Newton, and M. Rudnev, \emph{Higher convexity and iterated sum sets}, arXiv preprint, arXiv:2005.00125 (2020)

\bibitem{jonesolly} T. G. F. Jones and O. Roche-Newton, \emph{Improved bounds on the set A(A+1)}, J. Combin. Theory Ser. A,
120, 3, 515-526 (2013)

\bibitem{liolly}L. Li and O. Roche-Newton, \emph{Convexity and a sum-product type estimate}, Acta Arith. ,156(3), 247--255 (2012)

\bibitem{brendan-richlinesgrids} B. Murphy, \emph{Upper and lower bounds for rich lines in grids}, Amer. J. Math, 142, no.2. To appear. (2021)

\bibitem{affinegroup4} G. Petridis, O. Roche-Newton, M. Rudnev, and A. Warren, \emph{An energy bound in the affine group}, Int. Math. Res. Not. (2020)

\bibitem{Q} O. Roche-Newton and M. Rudnev, \emph{On the Minkowski distances and products of sum sets}, Israel J. Math, 209, no. 2, 507-526 (2015)

\bibitem{ollyaudie} O. Roche-Newton and A. Warren, \emph{New expander bounds from affine group energy}, Discrete \& Comput. Geom. (2020)

\bibitem{rss} M. Rudnev, G. Shakan, and I. D. Shkredov, \emph{Stronger sum-product inequalities for small sets}, Proc. Amer. Math. Soc. 148, 1467-1479  (2020)

\bibitem{rudnevshkredov} M. Rudnev and I. D. Shkredov, \emph{On growth rate in $SL_2(\F_p)$, the affine group and sum-product type applications}, arXiv preprint, arXiv:1812.01671 (2018)

\bibitem{energy} M. Rudnev, I. D. Shkredov and S. Stevens, 
{\em On the energy variant of the sum-product conjecture,}  Rev. Mat. Iberoam. {\bf 36}:1, 207--232 (2020)

\bibitem{MishaSophie} M. Rudnev and S. Stevens, \emph{An update on the sum-product problem}, arXiv e-prints, arXiv:2005.11145 (2020)

\bibitem{solymosiruzsa} I. Ruzsa and J. Solymosi, \emph{Sumsets of semiconvex sets}, arXiv preprint, arXiv: 2008.08021 (2020)

\bibitem{schoenshkredov} T. Schoen and I. D. Shkredov, \emph{On sumsets of convex sets}, Combin. Probab. Comput., 20 (5), 793-798 (2011)

\bibitem{shakan} G. Shakan, {\emph On higher energy decompositions and the sum-product phenomenon,} 
Math. Proc. Cambridge Phil. Soc. {\bf 167}:3,  599--617 (2019)

\bibitem{shkredovconvex} I. D. Shkredov, \emph{On sums of Szemerédi-Trotter sets}, Proc. Steklov Inst. Math., 289, 300–309 (2015)

\bibitem{shkredovbw}  I. D. Shkredov, {\em Some remarks on the Balog-Wooley decomposition theorem and quantities $D^+$, $D^\times$,} Proc. Steklov Inst. Math. {\bf 208}(Suppl 1):74, 74--90 (2017)

\bibitem{Solymosi} J. Solymosi, \emph{Bounding multiplicative energy by the sumset}. Adv. Math. {\bf{222}} no. 2, 402-408 (2009)


\bibitem{sophiefrank} S. Stevens and F. de Zeeuw, \emph{An improved point‐line incidence bound over arbitrary fields}, Bull. Lond. Math. Soc., 49: 842-858 (2017)

\bibitem{xue} B. Xue, \emph{Asymmetric estimates and the sum-product problems}, arXiv preprint, arXiv:2005.09893 (2020)

\end{thebibliography}
\end{document}